\documentclass[reqno]{amsart}

\newtheorem{Theorem}{Theorem}[section]
\newtheorem{Proposition}[Theorem]{Proposition}
\newtheorem{Corollary}[Theorem]{Corollary}

\theoremstyle{remark}
\newtheorem{Remark}[Theorem]{Remark}

\numberwithin{equation}{section}

\allowdisplaybreaks

\title{Shifted versions of the Bailey and Well-Poised Bailey Lemmas}

\author[Fr\'ed\'eric Jouhet]{Fr\'ed\'eric Jouhet$^*$}

\address{Universit\'e de Lyon, Universit\'e Lyon I \\
CNRS, UMR 5208 Institut Camille Jordan\\
B\^atiment du Doyen Jean Braconnier\\
43, bd du 11 Novembre 1918\\
 69622 Villeurbanne Cedex, France}
\email{jouhet@math.univ-lyon1.fr}
\urladdr{http://math.univ-lyon1.fr/~jouhet}


\begin{document}
\begin{abstract}
The Bailey lemma is a famous tool to prove Rogers-Ramanujan type identities. We use shifted versions of  the Bailey lemma to derive $m$-versions of multisum Rogers-Ramanujan type identities. We also apply this method to the Well-Poised Bailey lemma and obtain a new extension of the Rogers-Ramanujan identities.  
\end{abstract}

\maketitle

\section{Introduction}

The Rogers-Ramanujan identities
\begin{eqnarray}
\sum_{k=0}^\infty\frac{q^{k^2}}{(1-q)\cdots(1-q^k)}&=&\prod_{n\geq0}\frac{1}{(1-q^{5n+1})(1-q^{5n+4})},\label{rr1}\\
\sum_{k=0}^\infty\frac{q^{k^2+k}}{(1-q)\cdots(1-q^k)}&=&\prod_{n\geq0}\frac{1}{(1-q^{5n+2})(1-q^{5n+3})}\label{rr2}
\end{eqnarray}
are among the most famous $q$-series identities in partition theory and
combinatorics. Since their discovery they have been proved and generalized in various ways (see
\cite{A2, BIS, GIS} and the references cited there). A classical approach to get this kind of identities is the Bailey lemma, originally proved by Bailey \cite{Ba} and later strongly highlighted by Andrews \cite{A1, A2, AAR}. The goal of this paper is to use bilateral extensions of this tool to derive new generalizations of (\ref{rr1}) and (\ref{rr2}) as well as other famous identities of the same kind.\\
First, recall some standard notations for $q$-series which can be found in \cite{GR}. Let $q$ be a fixed complex parameter (the ``base'') with $0<|q|<1$.
The \emph{$q$-shifted factorial} is defined for any complex
parameter $a$ by
\begin{equation*}
(a)_\infty\equiv (a;q)_\infty:=\prod_{j\geq 0}(1-aq^j)\;\;\;\;\mbox{and}\;\;\;\;(a)_k\equiv (a;q)_k:=\frac{(a;q)_\infty}{(aq^k;q)_\infty},
\end{equation*}
where $k$ is any integer.
Since the same base $q$ is used throughout this paper,
it may be readily omitted (in notation, writing $(a)_k$ instead of $(a;q)_k$, etc) which will not lead to any confusion. For brevity, write
\begin{equation*}
(a_1,\ldots,a_m)_k:=(a_1)_k\cdots(a_m)_k,
\end{equation*}
where $k$ is an integer or infinity. The \emph{$q$-binomial coefficient} is defined as follows:
$$\left[{n\atop k}\right]_q:=\frac{(q)_n}{(q)_k(q)_{n-k}},$$
and we assume that $\left[{n\atop k}\right]_q=0$ if $k<0$ or $k>n$. Further, recall the \emph{basic hypergeometric series}
\begin{equation*}
{}_s\phi_{s-1}\!\left[\begin{matrix}a_1,\dots,a_s\\
b_1,\dots,b_{s-1}\end{matrix};q,z\right]:=
\sum_{k=0}^\infty\frac{(a_1,\dots,a_s)_k}{(q,b_1,\dots,b_{s-1})_k}z^k,
\end{equation*}
and the \emph{bilateral basic hypergeometric series} 
\begin{equation*}
{}_s\psi_s\!\left[\begin{matrix}a_1,\dots,a_s\\
b_1,\dots,b_s\end{matrix};q,z\right]:=
\sum_{k=-\infty}^\infty\frac{(a_1,\dots,a_s)_k}{(b_1,\dots,b_s)_k}z^k.
\end{equation*}
The set of nonnegative (resp. positive) integers will be denoted by $\mathbb{N}$ (resp.  $\mathbb{N}^*$). We will use along this paper the following results on $q$-series, which are the finite $q$-binomial \cite[Appendix, (II.4)]{GR}, $q$-Pfaff-Saalsch\"utz \cite[Appendix, (II.12)]{GR} and Jacobi triple product  \cite[Appendix, (II.28)]{GR} identities respectively:
\begin{eqnarray}
&&{}_1\phi_{0}\!\left[\begin{matrix}q^{-n}\\
-\end{matrix};q,z\right]=(zq^{-n})_n\;\;\mbox{for}\;\;n\in\mathbb{N},\label{qbi}\\
&&{}_3\phi_{2}\!\left[\begin{matrix}a,b,q^{-n}\\
c,abq^{1-n}/c\end{matrix};q,q\right]=\frac{(c/a,c/b)_n}{(c,c/ab)_n}\;\;\mbox{for}\;\;n\in\mathbb{N},\label{qps}\\
&&\sum_{n\in\mathbb{Z}}(-1)^nz^nq^{\left({n\atop 2}\right)}=(q,z,q/z)_\infty\label{jtp}.
\end{eqnarray}

Recall \cite{AAR} that a Bailey pair $(\alpha_n(a,q),\,\beta_n(a,q))$ related to $a$ and $q$ is defined by the relation: 
\begin{equation}\label{bp}
\beta_n(a,q)=\sum_{r=0}^n\frac{\alpha_r(a,q)}{(q)_{n-r}(aq)_{n+r}}\;\;\;\;\forall\,n\in\mathbb{N}.
\end{equation}
The Bailey lemma describes how, from a Bailey pair, one can produce infinitely many of them:
\begin{Theorem}[Bailey lemma]
If $(\alpha_n(a,q),\,\beta_n(a,q))$ is a Bailey pair related to $a$ and $q$, then so is $(\alpha'_n(a,q),\,\beta'_n(a,q))$, where
$$\alpha'_n(a,q)={(\rho_1,\rho_2)_n(aq/\rho_1\rho_2)^n\over
(aq/\rho_1,aq/\rho_2)_n}\,\alpha_n(a,q)$$
and
$$\beta'_n(a,q)=\sum_{j\geq
0}{(\rho_1,\rho_2)_j(aq/\rho_1\rho_2)_{n-j}(aq/\rho_1\rho_2)^j\over (q)_{n-j}(aq/\rho_1,aq/\rho_2)_n}\,\beta_j(a,q).$$
\end{Theorem}
In \cite{AAR}, the following unit Bailey pair is considered:
\begin{equation}\label{ubp}
\alpha_n={(-1)^nq^{n(n-1)/2}(a)_n(1-aq^{2n})\over (1-a)(q)_n},\qquad \beta_n=\delta_{n,0},
\end{equation}
and two iterations of Theorem~1.1 applied to \eqref{ubp} proves Watson's transformation \cite[Appendix, (III.18)]{GR}, which is a six  parameters finite extension of (\ref{rr1}) and (\ref{rr2}).\\

Now we want to point out that in the definition of a Bailey pair (\ref{bp}), the condition that the sum on the right-hand side must vanish for $r>n$ is ``natural'', in the sense that $1/(q)_{n-r}=0$ for $n-r<0$. However, the fact that this sum starts at $r=0$ can not be omitted, therefore the definition of a Bailey pair would be slightly different if the sum could start from $-\infty$ up to $n$. 
As noticed in \cite{BMS}, one can define for all $n\in\mathbb{Z}$ a \emph{bilateral Bailey pair} $(\alpha_n(a,q),\,\beta_n(a,q))$ related to $a$ and $q$ by the relation: 
\begin{equation}\label{bbp}
\beta_n(a,q)=\sum_{r\leq n}{\alpha_r(a,q)\over (q)_{n-r}(aq)_{n+r}}\;\;\;\;\forall\,n\in\mathbb{Z}.
\end{equation} 
It is of course possible to find bilateral Bailey pairs with general $a$, but it seems difficult to express $\beta_n(a,q)$ in a nice (closed) form. However we will see in the remainder of this paper that it becomes easier in the special case $a=q^m$, $m\in\mathbb{N}$, the reason being that the sum on the right-hand side of (\ref{bbp}) will run from $-m-n$ to $n$, and therefore will be finite. We found in this case more appropriate to call such a bilateral Bailey pair $(\alpha_n(q^m,q),\,\beta_n(q^m,q))$ a \emph{shifted Bailey pair}. \\
In \cite{BMS}, the Bailey lemma is extended in the following way: 
\begin{Theorem}[Bilateral Bailey lemma]
If $(\alpha_n(a,q),\,\beta_n(a,q))$ is a bilateral Bailey pair related to $a$ and $q$, then
so is $(\alpha'_n(a,q),\,\beta'_n(a,q))$, where
$$\alpha'_n(a,q)={(\rho_1,\rho_2)_n(aq/\rho_1\rho_2)^n\over
(aq/\rho_1,aq/\rho_2)_n}\alpha_n(a,q)$$
and
$$\beta'_n(a,q)=\sum_{j\leq n}{(\rho_1,\rho_2)_j(aq/\rho_1\rho_2)_{n-j}(aq/\rho_1\rho_2)^j\over (q)_{n-j}(aq/\rho_1,aq/\rho_2)_n}\beta_j(a,q),$$
subject to convergence conditions on the sequences $\alpha_n(a,q)$ and $\beta_n(a,q)$, which make the relevant infinite series absolutely convergent.
\end{Theorem}
\begin{Remark}
In Theorem~1.1, no problem occurs with changing summations as the sum in (\ref{bp}) is finite. This is not true any more in the bilateral version, therefore one needs to add absolute convergence conditions to change summations before using the $q$-Pfaff-Saalsch\"utz identity. Note also that these convergence conditions are not needed in the particular case $a=q^m$, $m\in\mathbb{N}$, of shifted Bailey pairs, which will be often used throughout the paper.
\end{Remark}

There is an extension of the Bailey lemma, the Well-Poised (or WP-) Bailey lemma \cite{AB}, which also has a bilateral version. Indeed, define for all $n\in\mathbb{Z}$ a \emph{WP-bilateral Bailey pair} $(\alpha_n(a,\alpha),\,\beta_n(a,\alpha))$ related to $a$ and $\alpha$ by the relation: 
\begin{equation}\label{wpbbp}
\beta_n(a,\alpha)=\sum_{r\leq n}\frac{(\alpha/a)_{n-r}(\alpha)_{n+r}}{(q)_{n-r}(aq)_{n+r}}\alpha_r(a,\alpha)\;\;\;\;\forall\,n\in\mathbb{Z}.
\end{equation} 
The results of \cite{AB} can also be extended to the bilateral case. We omit the proof, as it is exactly the same as in \cite{AB}: it requires Jackson's ${}_8\phi_7\!$ finite summation \cite[Appendix, (II. 22]{GR}, and the $q$-Pfaff-Saalsch\"utz identity \eqref{qps}.
\begin{Theorem}[WP-bilateral Bailey lemma]
If $(\alpha_n(a,\alpha),\,\beta_n(a,\alpha))$ is a WP-bilateral Bailey pair related to $a$ and $\alpha$, then so are 
$(\alpha'_n(a,\alpha),\,\beta'_n(a,\alpha))$ and $(\widetilde{\alpha}_n(a,\alpha),\,\widetilde{\beta}_n(a,\alpha))$, where

$$\alpha'_n(a,\alpha)={(\rho_1,\rho_2)_n\over
(aq/\rho_1,aq/\rho_2)_n}(\alpha/c)^n\alpha_n(a,c), $$

\begin{multline*}
\beta'_n(a,\alpha)=\frac{(\alpha\rho_1/a,\alpha\rho_2/a)_n}{(aq/\rho_1,aq/\rho_2)_n}\sum_{j\leq n}\frac{(\rho_1,\rho_2)_j}{(\alpha\rho_1/a,\alpha\rho_2/a)_j}\\
\times\frac{1-cq^{2j}}{1-c}
\frac{(\alpha/c)_{n-j}(\alpha)_{n+j}}{(q)_{n-j}(qc)_{n+j}}(\alpha/c)^j\beta_j(a,c),
\end{multline*}

with $c=\alpha\rho_1\rho_2/aq$, and

$$\widetilde{\alpha}_n(a,\alpha)=\frac{(qa^2/\alpha)_{2n}}{(\alpha)_{2n}}(\alpha^2/qa^2)^n\alpha_n(a,qa^2/\alpha), $$

$$\widetilde{\beta}_n(a,\alpha)=\sum_{j\leq n}\frac{(\alpha^2/qa^2)_{n-j}}{(q)_{n-j}}(\alpha^2/qa^2)^j\beta_j(a,qa^2/\alpha),$$
subject to convergence conditions on the sequences $\alpha_n$ and $\beta_n$, which make the relevant infinite series absolutely convergent.
\end{Theorem}
Note that if $\alpha=0$, then the first instance of the previous theorem reduces to Theorem~1.2. As before, we will often avoid convergence conditions by setting $a=q^m$, $m\in\mathbb{N}$, and such WP-bilateral Bailey pairs will be called \emph{WP-shifted Bailey pairs}.
\begin{Remark}
One can see that any shifted Bailey pair (resp. WP-shifted Bailey pair) is equivalent to a classical Bailey pair (resp. WP-Bailey pair) related to $a=1$ or $a=q$, according to the parity of $m$. Thus, the concept of (WP-) shifted Bailey pairs is nothing else but an appropriate and useful way of writing some (WP-) Bailey pairs, yielding surprising (known and new) identities.
\end{Remark}
We also want to point out that Schlosser proved a very general bilateral well-poised Bailey lemma, based on a matrix inversion \cite{Sc}, and which is different from Theorem~1.4.\\

This paper is organized as follows. In section~2, we give a shifted Bailey pair, which is used to prove in an elementary way $m$-versions of multisum Rogers-Ramanujan type identities. We will also point out some interesting special cases, including the $m$-versions of the Rogers-Ramanujan identities from \cite{GIS}. In Section~3 we give some results concerning bilateral versions of the change of base in Bailey pairs from \cite{BIS}, yielding $m$-versions of other multisum Rogers-Ramanujan type identities. In section~4, we first give a ``unit'' WP-bilateral Bailey pair which, by applying Theorem~1.4 yields a bilateral transformation generalizing both Ramanujan's ${}_1\psi_1\!$ and Bailey's ${}_6\psi_6\!$ summation formulae. We also find a WP-shifted Bailey pair, which yields a stricking extension of the Rogers-Ramanujan identities, generalizing some other results of \cite{GIS}. Finally, in the last section, we will give a few concluding remarks.

\section{A shifted Bailey pair and applications}

The following result gives a shifted Bailey pair, i.e., a bilateral Bailey pair related to $a=q^m$, $m\in\mathbb{N}$, which was already mentioned, but in another form, in \cite{ASW}, where the authors generalize this Bailey pair to the $A_2$ case.
\begin{Proposition}
For $m\in\mathbb{N}$, $(\alpha_n(q^m,q),\,\beta_n(q^m,q))$ is a shifted Bailey pair, where
$$\alpha_n(q^m,q)=(-1)^nq^{\left({n\atop 2}\right)}$$
and
$$\beta_n(q^m,q)=(q)_m(-1)^nq^{\left({n\atop 2}\right)}\left[{m+n\atop m+2n}\right]_q.$$
\end{Proposition}
\begin{proof}
We have by definition
$$\beta_n(q^m,q)=\sum_{k\leq n}\frac{(-1)^kq^{\left({k\atop 2}\right)}}{(q)_{n-k}(q^{1+m})_{n+k}}$$
so, as $1/(q^{1+m})_{n+k}=0$ if $n+k+m<0$, we can see that $\beta_n=0$ unless $2n+m\geq 0$. In that case, one has
$$
\beta_n(q^m,q)=\sum_{k\geq 0}\frac{(-1)^{n-k}q^{\left({n-k\atop 2}\right)}}{(q)_{k}(q^{1+m})_{2n-k}}=\frac{(-1)^nq^{\left({n\atop 2}\right)}}{(q^{1+m})_{2n}}\sum_{k\geq 0}\frac{(q^{-2n-m})_k}{(q)_{k}}\left(q^{m+n+1}\right)^k.
$$
As $2n+m\geq 0$, we can apply \eqref{qbi} to the sum over $k$. We get
$$
\beta_n(q^m,q)=(-1)^nq^{\left({n\atop 2}\right)}\frac{(q)_m}{(q)_{m+2n}}(q^{-n+1})_{m+2n}
=(-1)^nq^{\left({n\atop 2}\right)}(q)_m\frac{(q)_{m+n}}{(q)_{m+2n}(q)_{-n}},
$$
which is the desired result.
\end{proof}
\begin{Remark}
The special cases $m=0$ and $1$ in Proposition~2.1 correspond to the unit Bailey pair (\ref{ubp}) with $a=1$ and $q$. These two values of the parameter $a$ are in all classical uses of the Bailey lemma the only ones for which Jacobi triple product identity \eqref{jtp} can be used to get interesting Rogers-Ramanujan type identities. The clue in the present shifted case is that \eqref{jtp} can be used for all $a=q^m$, $m\in\mathbb{N}$.
\end{Remark}
We will need two instances of the bilateral Bailey lemma, which are given by specializing $\rho_1,\rho_2\to\infty$ and $\rho_1=\sqrt{aq},\rho_2\to\infty$ in Theorem~1.2 respectively:
\begin{equation}\label{s1}
\alpha'_n(a,q)=q^{n^2}a^n\alpha_n(a,q),\;\;\;\;\;\;\beta'_n(a,q)=\sum_{j\leq n}\frac{q^{j^2}a^j}{(q)_{n-j}}\beta_j(a,q),
\end{equation}
and 
\begin{equation}\label{s2}
\alpha'_n(a,q)=q^{n^2/2}a^{n/2}\alpha_n(a,q),\;\;\;\;\;\;\beta'_n(a,q)=\sum_{j\leq n}\frac{q^{j^2/2}a^{j/2}}{(q)_{n-j}}\frac{(-\sqrt{aq})_{j}}{(-\sqrt{aq})_{n}}\beta_j(a,q).
\end{equation} 

\noindent Now we can state a first consequence of Proposition~2.1.
\begin{Theorem}
For all $k\in\mathbb{N}^*$ and $m\in\mathbb{N}$ we have
\begin{multline}\label{kmrr}
\sum_{-\left\lfloor m/2\right\rfloor\leq n_k\leq n_{k-1}\leq\dots\leq n_1}\frac{q^{n_1^2+\dots+n_k^2+m(n_1+\dots+n_k)}}{(q)_{n_1-n_2}\dots(q)_{n_{k-1}-n_k}}(-1)^{n_k}q^{\left({n_k\atop 2}\right)}\left[{m+n_k\atop m+2n_k}\right]_q\\
=\frac{(q^{2k+1},q^{k(m+1)},q^{k(1-m)+1};q^{2k+1})_\infty}{(q)_\infty},
\end{multline}
and
\begin{multline}\label{kmgg}
\sum_{-\left\lfloor m/2\right\rfloor\leq n_k\leq n_{k-1}\leq\dots\leq n_1}\frac{q^{n_1^2/2+n_2^2+\dots+n_k^2+m(n_1/2+n_2+\dots+n_k)}(-q^{(m+1)/2})_{n_1}}{(q)_{n_1-n_2}\dots(q)_{n_{k-1}-n_k}}\\
\times(-1)^{n_k}q^{\left({n_k\atop 2}\right)}\left[{m+n_k\atop m+2n_k}\right]_q\\
=\frac{(-q^{(m+1)/2})_\infty}{(q)_\infty}\,(q^{2k},q^{(k-1/2)(m+1)},q^{k(1-m)+(m+1)/2};q^{2k})_\infty.
\end{multline}
\end{Theorem}
\begin{proof}
We apply $k$ times the instance (\ref{s1}) of the bilateral Bailey lemma with $a=q^m$ to our shifted Bailey pair $(\alpha_n(q^m,q),\beta_n(q^m,q))$ of Proposition~2.1, so we get a shifted Bailey pair $(\alpha^{(k)}_n(q^m,q),\beta^{(k)}_n(q^m,q))$, where
$$\alpha^{(k)}_n(q^m,q)=q^{kn^2+kmn}\alpha_n(q^m,q)$$
and
$$\beta^{(k)}_n(q^m,q)=\sum_{n_k\leq n_{k-1}\leq\dots\leq n_1\leq n}\frac{q^{n_1^2+\dots+n_k^2+m(n_1+\dots+n_k)}}{(q)_{n-n_1}(q)_{n_1-n_2}\dots(q)_{n_{k-1}-n_k}}\beta_{n_k}(q^m,q).$$
Invoking Tannery's Theorem \cite{Br} to interchange limit and summation, \eqref{kmrr} follows by letting $n\to+\infty$ in the relation
$$\beta^{(k)}_n(q^m,q)=\sum_{j\leq n}{\alpha^{(k)}_j(q^m,q)\over (q)_{n-j}(q^{1+m})_{n+j}}$$
and finally using (\ref{jtp}) to factorize the right-hand side.\\
For \eqref{kmgg}, we apply $k-1$ times the instance (\ref{s1}) of the bilateral Bailey lemma with $a=q^m$ to our shifted Bailey pair $(\alpha_n(q^m,q),\beta_n(q^m,q))$, and then once the instance (\ref{s2}), so we get a bilateral Bailey pair $(\alpha^{(k)}_n(q^m,q),\beta^{(k)}_n(q^m,q))$, where
$$\alpha^{(k)}_n(q^m,q)=q^{(k-1/2)n^2+(k-1/2)mn}\alpha_n(q^m,q)$$
and
\begin{multline*}
\beta^{(k)}_n(q^m,q)=\sum_{n_k\leq n_{k-1}\leq\dots\leq n_1\leq n}\frac{q^{n_1^2/2+n_2^2+\dots+n_k^2+m(n_1/2+n_2+\dots+n_k)}}{(q)_{n-n_1}(q)_{n_1-n_2}\dots(q)_{n_{k-1}-n_k}}\\
\times\frac{(-q^{(m+1)/2})_{n_1}}{(-q^{(m+1)/2})_{n}}\beta_{n_k}(q^m,q).
\end{multline*}
The result follows as before by letting $n\to+\infty$ in the relation
$$\beta^{(k)}_n=\sum_{j\leq n}{\alpha^{(k)}_j\over (q)_{n-j}(q^{1+m})_{n+j}}$$
and finally using (\ref{jtp}) to factorize the right-hand side.
\end{proof}
\begin{Remark}
Identity \eqref{kmrr} (resp. \eqref{kmgg}) is an $m$-version of the Andrews-Gordon identities (resp. the generalized G\"ollnitz-Gordon identities), which are obtained by setting $m=0$ and $1$ in \eqref{kmrr} (resp. $m=0$ and $2$ in \eqref{kmgg}). However, we do not get here $m$-versions of the \emph{full} Andrews-Gordon or G\"ollnitz-Gordon identities (see for instance \cite[p. 111]{A76} and \cite{BIS}). 
\end{Remark}

In the case $k=1$, we derive the following interesting identities.
\begin{Corollary}
For all $m\in\mathbb{N}$, we have:
\begin{eqnarray}
&&\hskip-1.2cm\sum_{j=0}^{\left\lfloor m/2\right\rfloor}(-1)^jq^{\left({j\atop 2}\right)}\left[{m-j\atop j}\right]_q=\left\{\begin{array}{ll}
(-1)^{\left\lfloor m/3\right\rfloor}q^{m(m-1)/6}&\mbox{if}\;\;m \not\equiv 2\;\mbox{(mod 3)},\\
0&\mbox{if}\;\;m \equiv 2\;\mbox{(mod 3)},
\end{array}\right.\label{k=1mrr}\\
&&\hskip-1.2cm\sum_{j=0}^{m}(-1)^jq^{2\left({j\atop 2}\right)}\left[{2m-j\atop j}\right]_{q^2}(-q;q^2)_{m-j}=
(-1)^{\left\lfloor m/2\right\rfloor}q^{m(3m-1)/2},\label{k=1mggpair}\\
&&\hskip-1.2cm\sum_{j=0}^{m}(-1)^jq^{\left({j\atop 2}\right)}\left[{2m+1-j\atop j}\right]_{q}(-q)_{m-j}=\left\{\begin{array}{ll}
(-1)^{\left\lfloor m/2\right\rfloor}q^{m(3m+2)/4}&\mbox{if}\;\;m\; \mbox{even},\\
0&\mbox{if}\;\;m\; \mbox{odd}.
\end{array}\right.\label{k=1mggimpair}
\end{eqnarray}
\end{Corollary}
\begin{proof}
In \eqref{kmrr}, take $k=1$ and replace the single index of summation by $-j$ to get:
$$
\sum_{j=0}^{\left\lfloor m/2\right\rfloor}(-1)^jq^{3\left({j\atop 2}\right)-(m-2)j}\left[{m-j\atop j}\right]_q=\frac{(q^3,q^{m+1},q^{2-m};q^3)_\infty}{(q)_\infty}.
$$
Write the right-hand side as
$$\frac{(q^{m+1},q^{2-m};q^3)_\infty}{(q,q^2;q^3)_\infty}=\left\{\begin{array}{ll}
(-1)^{\left\lfloor m/3\right\rfloor}q^{-m(m-1)/6}&\mbox{if}\;\;m \not\equiv 2\;\mbox{(mod 3)},\\
0&\mbox{if}\;\;m \equiv 2\;\mbox{(mod 3)},
\end{array}\right.,$$
then replace $q$ by $q^{-1}$ and use $\displaystyle\left[{n\atop k}\right]_{q^{-1}}=q^{k(k-n)}\left[{n\atop k}\right]_{q}$ to get \eqref{k=1mrr}.\\
Now \eqref{kmgg} with $k=1$ and $q$ replaced by $q^2$ can be rewritten:
\begin{equation}\label{interm}
\hskip-0.1cm\sum_{j=0}^{\left\lfloor m/2\right\rfloor}(-1)^jq^{4\left({j\atop 2}\right)-(m-3)j}\left[{m-j\atop j}\right]_{q^2}\frac{(-q^{m+1};q^2)_{-j}}{(-q^{m+1};q^2)_\infty}=\frac{(q^4,q^{m+1},q^{3-m};q^4)_\infty}{(q^2;q^2)_\infty}\cdot
\end{equation}
Next the $m$ even and odd cases have to be considered separately to simplify \eqref{interm}. Replace first $m$ by $2m$, multiply both sides by $(-q;q^2)_\infty$ and write the right-hand side as:
$$\frac{(q^{2m+1},q^{3-2m};q^4)_\infty}{(q,q^3;q^4)_\infty}=(-1)^{\left\lfloor m/2\right\rfloor}q^{-m(m-1)/2},$$
where the equality is obtained by considering the parity of $m$. Replacing $q$ by $q^{-1}$ yields \eqref{k=1mggpair} after a few simplications. \\
Finally, if we replace $m$ by $2m+1$ in \eqref{interm}, multiply both sides by $(-q^2;q^2)_\infty$ and write the right-hand side as:
$$\frac{(q^{2m+2},q^{2-2m};q^4)_\infty}{(q^2,q^2;q^4)_\infty}=\left\{\begin{array}{ll}
(-1)^{\left\lfloor m/2\right\rfloor}q^{-m^2/2}&\mbox{if}\;\;m \;\mbox{even},\\
0&\mbox{if}\;\;m \;\mbox{odd},
\end{array}\right.$$
then we obtain \eqref{k=1mggimpair} after replacing $q$ by $q^{-1/2}$ and simplifying.
\end{proof}
\begin{Remark}
Identity (\ref{k=1mrr}) is a well-known polynomial analogue of Euler's pentagonal number theorem which has been generalized to a multivariable version by Guo and Zeng in \cite{GZ2}, and extensively studied in the framework of $q$-Fibonacci polynomials by Cigler in \cite{Ci}. In \cite[Corollary~4.13]{Wa2}, Warnaar generalizes  (\ref{k=1mrr}) to a cubic summation formula for elliptic hypergeometric series. Identity \eqref{k=1mggpair} is a hidden special case of the terminating $q$-analogue of Whipple's ${}_3F_2$ sum \cite[Appendix, (II.19)]{GR}. Finally, \eqref{k=1mggimpair} is a special case of an identity obtained by Gessel and Stanton through $q$-Lagrange inversion, generalized to the elliptic case by Warnaar in \cite[Corollary~4.11]{Wa2}.
\end{Remark}
Now we study further the case $k=2$ of (\ref{kmrr}) and (\ref{kmgg}).
\begin{Corollary}
For all $m\in\mathbb{N}$, we have:
\begin{multline}\label{mrr}
\sum_{j\geq 0}(-1)^jq^{5\left({j\atop 2}\right)-(2m-3)j}\left[{m-j\atop j}\right]_q\sum_{k\geq 0}\frac{q^{k^2+(m-2j)k}}{(q)_{k}}\\
=\frac{(q^5,q^{2m+2},q^{3-2m};q^5)_\infty}{(q)_\infty},
\end{multline}
and
\begin{multline}\label{mgg}
\sum_{j\geq 0}(-1)^jq^{8\left({j\atop 2}\right)-(3m-4)j}\left[{m-j\atop j}\right]_{q^2}\sum_{k\geq 0}q^{k^2+(m-2j)k}\frac{(-q^{m+1};q^2)_{k}}{(q^2;q^2)_{k}}\\
=\frac{(-q^{m+1};q^2)_\infty}{(q^2;q^2)_\infty}(q^8,q^{3m+3},q^{5-3m};q^8)_\infty.
\end{multline}
\end{Corollary}
\begin{proof}
Take $k=2$ in \eqref{kmrr}, then the left-hand side, after a few rearrangements, is equal to:
$$\sum_{j\geq 0}(-1)^jq^{5\left({j\atop 2}\right)-(2m-3)j}\left[{m-j\atop j}\right]_q\sum_{k\geq 0}\frac{q^{k^2+(m-2j)k}}{(q)_{k}},$$
and this yields \eqref{mrr}. For \eqref{mgg}, let $k=2$ in (\ref{kmgg}), simplify as before and then replace $q$ by $q^2$.
\end{proof}

Identity \eqref{mrr} is an $m$-version of the Rogers-Ramanujan identities, which was discovered by Garrett, Ismail and Stanton in \cite[Theorem 3.1]{GIS} using the theory of $q$-orthogonal polynomials and integral evaluation. The authors derived with the same method the following identity: 
\begin{equation}\label{gis}
\sum_{n\geq 0}\frac{q^{n^2+nm}}{(q)_n}=\frac{1}{(q)_\infty}\sum_{k=0}^m\left[{m\atop k}\right]_qq^{2k(k-m)}(q^5,q^{3+4k-2m},q^{2-4k+2m};q^5)_\infty.
\end{equation}
Note that \eqref{gis} is a famous $m$-version of the Rogers-Ramanujan identities, which is the inverse of \eqref{mrr}. Other identities related to (\ref{gis}) are proved in \cite{W}. To our knowledge, the $m$-version \eqref{mgg} of the G\"ollnitz-Gordon identities seems to be new.\\
In view of \eqref{mrr} and \eqref{gis}, it is possible to invert (\ref{kmrr}) through the classical Bailey inversion (see for instance \cite{AAR}). This is done in the following theorem, which is a $k$-generalization of \eqref{mrr}. Unfortunately, it seems not possible to get in the same way nice inversions of \eqref{kmgg} or \eqref{mgg}.
\begin{Theorem}
For all $k\in\mathbb{N}^*$ and $m\in\mathbb{N}$ we have
\begin{multline}\label{kmrrinvert}
\sum_{0\leq n_{k-1}\leq\dots\leq n_1}\frac{q^{n_1^2+\dots+n_{k-1}^2+m(n_1+\dots+n_{k-1})}}{(q)_{n_1-n_2}\dots(q)_{n_{k-2}-n_{k-1}}(q)_{n_{k-1}}}=\sum_{j=0}^m\left[{m\atop j}\right]_qq^{kj(j-m)}\\
\times\frac{(q^{2k+1},q^{k(m-2j+1)},q^{k(1-m+2j)+1};q^{2k+1})_\infty}{(q)_\infty}\cdot
\end{multline}
\end{Theorem}
\begin{proof}
We will only consider the even case where $m$ is replaced by $2m$, the process is the same in the odd case. Shift $n_i\to n_i-n_k$ for $1\leq i\leq k-1$,  set $j=-n_k$ and finally replace $j$ by $m-j$ to get
$$a_m=\frac{1-q^{2m+1}}{1-q}\sum_{j=0}^m(-1)^{m-j}q^{\left({m-j\atop 2}\right)}\frac{(q)_{m+j}}{(q)_{m-j}}b_j,$$
where
$$a_m:=q^{km^2-m} \frac{1-q^{2m+1}}{1-q}\frac{(q^{2k+1},q^{k(2m+1)},q^{k(1-2m)+1};q^{2k+1})_\infty}{(q)_\infty},$$
and
$$b_m:=\frac{q^{km^2-m}}{(q)_{2m}}\sum_{0\leq n_{k-1}\leq\dots\leq n_1}\frac{q^{n_1^2+\dots+n_{k-1}^2+2m(n_1+\dots+n_{k-1})}}{(q)_{n_1-n_2}\dots(q)_{n_{k-2}-n_{k-1}}(q)_{n_{k-1}}}\cdot$$
The classical Bailey inversion \cite{AAR} gives $\displaystyle b_m=\sum_{j=0}^m\frac{a_j}{(q)_{m-j}(q^2)_{m+j}}$, which can be rewritten
\begin{multline*}
b_m=\frac{q^{km^2-m}}{(q)_{2m}}\sum_{j=0}^m\frac{q^j-q^{2m-j+1}}{1-q^{2m-j+1}}\left[{2m\atop j}\right]_qq^{kj^2-2kjm}\\
\times\frac{(q^{2k+1},q^{k(2m-2j+1)},q^{k(1-2m+2j)+1};q^{2k+1})_\infty}{(q)_\infty}\cdot
\end{multline*}
Writing $\displaystyle\frac{q^j-q^{2m-j+1}}{1-q^{2m-j+1}}=1+\frac{q^j-1}{1-q^{2m-j+1}}$, splitting the sum over $j$ into two parts, and replacing $j$ by $2m+1-j$ in the second sum, the resulting identity is \eqref{kmrrinvert}.
\end{proof}
In \cite{G}, Garrett obtained $m$-versions of the full Andrews-Gordon identities, thus generalizing \eqref{kmrrinvert}. Besides, Berkovich and Paule prove with another method in \cite[(3.21)]{BP} a \emph{negative $m$-version} of the full Andrews-Gordon identities. Warnaar  also obtained other identities of the same kind in \cite{Wa4}, by using a different approach from ours (although related to the Bailey lemma). It could be interesting to derive all these results of \cite{BP, G, Wa4} from our approach. A bilateral version of the famous Bailey lattice \cite{AAB} would probably be needed, and we will come back to these questions in a forthcoming paper. Before ending this Section, we note that (\ref{kmrr}) is in fact closely related to the full Andrews-Gordon identities. Indeed, replace $m$ by $2m$, and then shift $n_i\to n_i-m$ for $1\leq n_i\leq k$ in the left-hand side of (\ref{kmrr}). Using 
$$(q^{k(2m+1)},q^{k(1-2m)+1};q^{2k+1})_\infty=(-1)^mq^{-km^2+\left({m+1\atop 2}\right)}(q^{k+m+1},q^{k-m};q^{2k+1})_\infty$$
yields
\begin{multline}\label{k2mrr}
\sum_{0\leq n_k\leq n_{k-1}\leq\dots\leq n_1}\frac{q^{n_1^2+\dots+n_k^2}}{(q)_{n_1-n_2}\dots(q)_{n_{k-1}-n_k}}(-1)^{n_k}q^{\left({n_k\atop 2}\right)-mn_k}\left[{m+n_k\atop 2n_k}\right]_q\\
=\frac{(q^{2k+1},q^{k+m+1},q^{k-m};q^{2k+1})_\infty}{(q)_\infty}\cdot
\end{multline}
Notice that the right-hand side of (\ref{k2mrr}) is the same as in the full Andrews-Gordon identities, thus identifying the right-hand sides yields for all $k\in\mathbb{N}^*$ and $m\in\{1,\dots,k-1\}$:
\begin{multline}\label{lhsfullag}
\sum_{0\leq n_k\leq n_{k-1}\leq\dots\leq n_1}\frac{q^{n_1^2+\dots+n_k^2}}{(q)_{n_1-n_2}\dots(q)_{n_{k-1}-n_k}}(-1)^{n_k}q^{\left({n_k\atop 2}\right)-mn_k}\left[{m+n_k\atop 2n_k}\right]_q\\
=\sum_{0\leq n_{k-1}\leq\dots\leq n_1}\frac{q^{n_1^2+\dots+n_{k-1}^2+n_{k-m}+\dots+n_{k-1}}}{(q)_{n_1-n_2}\dots(q)_{n_{k-1}}}\cdot
\end{multline}
Proving directly \eqref{lhsfullag} (i.e. without appealing to the full Andrews-Gordon identities) does not seem to be obvious. \\
The same link can be done between (\ref{kmgg}) and the full G\"ollnitz-Gordon identities. 

\section{Change of base}

In \cite{BIS}, many multisums of Rogers-Ramanujan type are proved as consequences of change of base in Bailey pairs. In the same vein as Section~2, many results concerning Bailey pairs in \cite{BIS} have bilateral versions. Here we will only highlight the following bilateral version of \cite[Theorem 2.1]{BIS}.
\begin{Theorem}
If $(\alpha_n(a,q),\,\beta_n(a,q))$ is a bilateral Bailey pair related to $a$ and $q$, then so is $(\alpha'_n(a,q),\,\beta'_n(a,q))$, where
$$\alpha'_n(a,q)=\frac{(-b)_n}{(-aq/b)_n}b^{-n}q^{-\left({n\atop 2}\right)}\alpha_n(a^2,q^2)$$
and
$$\beta'_n(a,q)=\sum_{k\leq n}\frac{(-aq)_{2k}(b^2;q^2)_k(q^{-k}/b,bq^{k+1})_{n-k}}{(b,-aq/b)_n(q^2;q^2)_{n-k}}b^{-k}q^{-\left({k\atop 2}\right)}\beta_k(a^2,q^2),$$
provided the relevant series are absolutely convergent.
\end{Theorem}
\begin{proof}
As in \cite{BIS}, we only need to use the definition (\ref{bbp}) of a bilateral Bailey pair, interchange summations and apply Singh's quadratic transformation \cite[Appendix, (III.21)]{GR} summed with $q$-Pfaff-Saalsch\"utz \eqref{qps}.
\end{proof}
The following result gives an $m$-version of Bressoud's identities for even moduli:
\begin{Theorem}
For all integers $m\in\mathbb{N}$ and $k\geq 1$ we have
\begin{multline}\label{kmrrchange}
\sum_{-\left\lfloor m/2\right\rfloor\leq n_k\leq n_{k-1}\leq\dots\leq n_1}\frac{q^{n_1^2+\dots+n_k^2+m(n_1+\dots+n_{k-1})+n_{k-1}-2n_k}(-q)_{2n_k+m}}{(q)_{n_1-n_2}\dots(q)_{n_{k-2}-n_{k-1}}(q^2;q^2)_{n_{k-1}-n_k}}\\
\times(-1)^{n_k}\left[{m+n_k\atop m+2n_k}\right]_{q^2}\\
=\frac{(q^{2k},q^{(k-1)(m+1)},q^{(k-1)(1-m)+2};q^{2k})_\infty}{(q)_\infty}.
\end{multline}
\end{Theorem}
\begin{proof}
Specialize $b\to\infty$ in Theorem~3.1:
\begin{equation}\label{d1}
\alpha'_n(a,q)=\alpha_n(a^2,q^2),\;\;\;\;\;\;\beta'_n(a,q)=\sum_{j\leq n}\frac{(-aq)_{2j}}{(q^2;q^2)_{n-j}}q^{n-j}\beta_j(a^2,q^2).
\end{equation}
Apply (\ref{d1}) to the shifted Bailey pair from Proposition~2.1. This gives a new shifted Bailey pair $(\alpha'_n(q^m,q),\beta'_n(q^m,q))$, where
$$\alpha'_n(q^m,q)=\alpha_n(q^{2m},q^2)=(-1)^nq^{2\left({n\atop 2}\right)}$$
and
\begin{eqnarray*}
\beta'_n(q^m,q)&=&\sum_{j\leq n}\frac{(-q^{1+m})_{2j}}{(q^2;q^2)_{n-j}}q^{n-j}\beta_j(q^{2m},q^2)\\
&=&(q^2;q^2)_{m}\sum_{j\leq n}\frac{(-q^{1+m})_{2j}}{(q^2;q^2)_{n-j}}q^{n-j}(-1)^jq^{2\left({j\atop 2}\right)}\left[{m+j\atop m+2j}\right]_{q^2}.
\end{eqnarray*}
Next apply $k-1$ times the instance (\ref{s1}) of the bilateral Bailey lemma to the new pair $(\alpha'_n(q^m,q),\beta'_n(q^m,q))$, this gives a shifted Bailey pair $(\alpha^{(k)}_n(q^m,q),\beta^{(k)}_n(q^m,q))$, where
$$\alpha^{(k)}_n(q^m,q)=q^{(k-1)n^2+(k-1)mn}(-1)^nq^{2\left({n\atop 2}\right)}$$
and
\begin{multline*}
\beta^{(k)}_n(q^m,q)=\sum_{n_k\leq n_{k-1}\leq\dots\leq n_1\leq n}\frac{q^{n_1^2+\dots+n_k^2+m(n_1+\dots+n_{k-1})+n_{k-1}-2n_k}(-q^{1+m})_{2n_k}}{(q)_{n-n_1}(q)_{n_1-n_2}\dots(q)_{n_{k-2}-n_{k-1}}(q^2;q^2)_{n_{k-1}-n_k}}\\
\times(q^2;q^2)_{m}(-1)^{n_k}\left[{m+n_k\atop m+2n_k}\right]_{q^2}.
\end{multline*}
Writing $\displaystyle(-q^{1+m})_{2n_k}=\frac{(-q)_{2n_k+m}}{(-q)_m}$, the result then follows by letting $n\to+\infty$ and invoking Tannery's theorem \cite{Br} in the relation
$$\beta^{(k)}_n=\sum_{j\leq n}{\alpha^{(k)}_j\over (q)_{n-j}(q^{1+m})_{n+j}}$$
and finally using (\ref{jtp}) to factorize the right-hand side.
\end{proof}

\begin{Remark}
The case $k=1$ of \eqref{kmrrchange} is trivial, while the case $k=2$, after a few simplications, appears to yield exactly identities \eqref{k=1mggpair} and \eqref{k=1mggimpair}.
\end{Remark}
As in Theorem~2.8, if we invert \eqref{kmrrchange} by using the classical Bailey inversion, then we obtain the following result.
\begin{Theorem}
For all $k\in\mathbb{N}^*$ and $m\in\mathbb{N}$ we have
\begin{multline}\label{kmgginvert}
\sum_{0\leq n_{k-1}\leq\dots\leq n_1}\frac{q^{n_1^2+\dots+n_{k-1}^2+m(n_1+\dots+n_{k-1})+n_{k-1}}}{(q)_{n_1-n_2}\dots(q)_{n_{k-2}-n_{k-1}}(q^2;q^2)_{n_{k-1}}}=\sum_{j=0}^m\left[{m\atop j}\right]_{q^2}\frac{q^{(k-1)j(j-m)}}{(-q)_m}\\
\times\frac{(q^{2k},q^{(k-1)(m-2j+1)},q^{(k-1)(1-m+2j)+2};q^{2k})_\infty}{(q)_\infty}\cdot
\end{multline}
\end{Theorem}
Note that \eqref{kmgginvert} is a special case of the results in \cite{G}, where $m$-versions of the \emph{full} Bressoud identities for even moduli (see for instance \cite{BIS}) are obtained. As at the end of the previous Section, by replacing $m$ by $2m$ and simplifying, it is possible to see that \eqref{kmrrchange} is related to the full Bressoud identities. 

\section{Applications of the bilateral WP-Bailey lemma}

As in \cite{W3}, it is possible to invert the relation \eqref{wpbbp} by using a matrix inversion, which gives:
\begin{equation}\label{invertwpbbp}
\alpha_n(a,\alpha,q)=\frac{1-aq^{2n}}{1-a}\sum_{r\leq n}\frac{(a)_{n+r}}{(q)_{n-r}}\frac{(a/\alpha)_{n-r}}{(\alpha q)_{n+r}}\frac{1-\alpha q^{2r}}{1-\alpha}\left(\frac{\alpha}{a}\right)^{n-r}\beta_r(a,\alpha,q),
\end{equation}
for all $n\in\mathbb{Z}$. Set $m\in\mathbb{N}$, then the following form a WP-bilateral Bailey pair:
\begin{equation}\label{uwpbbp}
\left\{\begin{array}{l}\displaystyle\alpha_n(a,\alpha,q)=\frac{1-aq^{2n}}{1-a}\frac{(a)_{n-m}}{(q)_{n+m}}\frac{(a/\alpha)_{n+m}}{(\alpha q)_{n-m}}\frac{1-\alpha q^{-2m}}{1-\alpha}\left(\frac{\alpha}{a}\right)^{n+m}\\
\displaystyle\beta_n(a,\alpha,q)=\delta_{n+m,0}.\end{array}\right.
\end{equation}
In the case $m=0$, we recover the unit WP-Bailey pair from \cite{AB}. Recall that two iterations of the first instance of the WP-Bailey lemma applied to the unit WP-Bailey pair gives Bailey's transformation between two terminating very-well poised ${}_{10}\phi_9\!$  \cite[Appendix, (III.28)]{GR}. We will show that two iterations of the WP-bilateral Bailey lemma to \eqref{uwpbbp} yields an extension of both Ramanujan's ${}_1\psi_1\!$ summation \cite[Appendix, (II.29)]{GR}, and Bailey's ${}_6\psi_6\!$ summation \cite[Appendix, (II.33)]{GR} formulae. This is stated in the following result.
\begin{Proposition}
For $m\in\mathbb{N}$, $|\alpha/a|<1$ and $|aq/\mu_1\mu_2|<1$, we have:
\begin{multline}\label{transf8psi8}
{}_8\psi_8\!\left[\begin{matrix}q\sqrt{a},-q\sqrt{a},\rho_1,\rho_2,\mu_1,\mu_2,aq^{-m},a^3q^{2+m}/\alpha\rho_1\rho_2\mu_1\mu_2\\
\sqrt{a},-\sqrt{a},aq/\rho_1,aq/\rho_2,aq/\mu_1,aq/\mu_2,q^{1+m},\alpha\rho_1\rho_2\mu_1\mu_2q^{-m}/qa^2\end{matrix};
q,\frac{\alpha}{a}\right]\\
=\frac{(aq,\lambda q/\mu_1,\lambda q/\mu_2,aq/\mu_1\mu_2)_\infty}{(\alpha/a,aq/\mu_1,aq/\mu_2,\lambda q)_\infty}\times\frac{(q/a,aq/\lambda\rho_1,aq/\lambda\rho_2,aq/\rho_1\rho_2)_m}{(q/\rho_1,q/\rho_2,q/\lambda,qa^2/\lambda\rho_1\rho_2)_m}\\
\times{}_8\psi_8\!\left[\begin{matrix}q\sqrt{\lambda},-q\sqrt{\lambda},\mu_1,\mu_2,\lambda\rho_1/a,\lambda\rho_2/a,\lambda q^{-m},aq^{1+m}/\rho_1\rho_2\\
\sqrt{\lambda},-\sqrt{\lambda},\lambda q/\mu_1,\lambda q/\mu_2,aq/\rho_1,aq/\rho_2,q^{1+m},\lambda\rho_1\rho_2q^{-m}/a\end{matrix};
q,\frac{\alpha}{\lambda}\right],
\end{multline}
where $\lambda:=\alpha\mu_1\mu_2/aq$.
\end{Proposition}
\begin{proof}
Apply twice the first instance of the WP-bilateral Bailey lemma to \eqref{uwpbbp}, this gives a WP-bilateral Bailey pair with four new parameters $\rho_1$, $\rho_2$, $\mu_1$ and $\mu_2$. Replacing it in \eqref{wpbbp} yields \eqref{transf8psi8} after letting  $n\to+\infty$ under the necessary conditions $|\alpha/a|<1$ and $|aq/\mu_1\mu_2|<1$ to use Tannery's theorem \cite{Br}, and simplifying.
\end{proof}
Now in \eqref{transf8psi8}, set $\alpha=0$, $\mu_1=b$, $\mu_2=aq/bz$, $\rho_1=aq/c$, $\rho_2=bz$, and finally $a=b$ and $m\to+\infty$. This gives after using on the left-hand side a limit case of the terminating very-well poised ${}_6\phi_5\!$ summation formula \cite[Appendix, (II-20)]{GR}:
\begin{equation}\label{1psi1}
\frac{(q,c/b,bz,q/bz)_\infty}{(c,q/b,z,c/bz)_\infty}={}_1\psi_1\!\left[\begin{matrix}b\\c\end{matrix};q,z\right],
\end{equation}
which is Ramanujan's ${}_1\psi_1\!$ summation formula, valid for $|q|<1$ and $|c/b|<|z|<1$.\\
Next, setting  in \eqref{transf8psi8} $\rho_1=b$, $\rho_2=c$, $\mu_2=e$, and $\mu_1=d=aq/\alpha$ (which gives $\lambda=e$) yields after letting $m\to+\infty$ and using on the right-hand side the same limit case of the terminating very-well poised ${}_6\phi_5\!$ summation formula \cite[Appendix, (II-20)]{GR}:
\begin{multline}\label{6psi6}
{}_6\psi_6\!\left[\begin{matrix}q\sqrt{a},-q\sqrt{a},b,c,d,e\\
\sqrt{a},-\sqrt{a},aq/b,aq/c,aq/d,aq/e\end{matrix};
q,\frac{qa^2}{bcde}\right]\\
=\frac{(q,aq,q/a,aq/bc,aq/bd,aq/be,aq/cd,aq/ce,aq/de)_\infty}
{(q/b,q/c,q/d,q/e,aq/b,aq/c,aq/d,aq/e,a^2q/bcde)_\infty},
\end{multline}
which is Bailey's ${}_6\psi_6\!$ summation, valid for $|a^2q/bcde|<1$.
\begin{Remark}
Note that by shifting the index of summation $k\to k-m$ on both sides of \eqref{transf8psi8}, one recovers the instance $n\to+\infty$ of Bailey's ${}_{10}\phi_9\!$ transformation formula \cite[Appendix, (III.28)]{GR}, which corresponds to \cite[Appendix, (III.23)]{GR}. Now if $m\to+\infty$ in \eqref{transf8psi8}, the left-hand side is independant of $\alpha$, and we recover a ${}_{6}\psi_6\!$ transformation formula from \cite{JS}, which can be iterated to yield directly \eqref{6psi6}, without appealing to the ${}_6\phi_5\!$ summation formula. 
\end{Remark}

In what follows, we give a new WP-shifted Bailey pair (recall that this means setting $a=q^m$, $m\in\mathbb{N}$ to avoid convergence problems), which generalizes the shifted Bailey pair from Section~2:
\begin{Proposition}
For $m\in\mathbb{N}$, $(\alpha_n(q^m,\alpha),\,\beta_n(q^m,\alpha))$ is a WP-shifted Bailey pair related to $a=q^m$ and $\alpha$, where
$$\alpha_n(q^m,\alpha)=\frac{(q^{m}/\alpha)_n}{(\alpha q^{-m})_n}(\alpha q^{-m})^n$$
and
$$\beta_n(q^m,\alpha)=\frac{(q)_{m}(q/\alpha)_{m-n}(\alpha^2q^{-2m})_{m+2n}}{(q/\alpha,\alpha q^{-m})_{m}(\alpha q^{1-m})_{m+n}}\left[{m+n\atop m+2n}\right]_q(q^m/\alpha)^{n}.$$
\end{Proposition}
\begin{proof}
We have by definition
$$\beta_n(q^m,\alpha)=\sum_{r\leq n}\frac{(\alpha q^{-m})_{n-r}(\alpha)_{n+r}}{(q)_{n-r}(q^{1+m})_{n+r}}\alpha_r(q^m,\alpha)$$
so, as $1/(q^{1+m})_{n+r}=0$ if $n+r+m<0$, we can see that $\beta_n=0$ unless $2n+m\geq 0$. In that case, one has
\begin{eqnarray*}
\beta_n(q^m,\alpha)&=&\sum_{k\geq 0}\frac{(\alpha q^{-m})_{k}(\alpha)_{2n-k}(q^{m}/\alpha)_{n-k}}{(q)_{k}(q^{1+m})_{2n-k}(\alpha q^{-m})_{n-k}}(\alpha q^{-m})^{n-k}\\
&=&\frac{(\alpha)_{2n}(q^{m}/\alpha)_{n}}{(q^{1+m})_{2n}(\alpha q^{-m})_{n}}(\alpha q^{-m})^n\sum_{k\geq 0}\frac{(\alpha q^{-m},q^{-2n-m},q^{1-n+m}/\alpha)_k}{(q,q^{1-2n}/\alpha,\alpha q^{1-n-m})_{k}}q^k.
\end{eqnarray*}
As $2n+m\geq 0$, the last sum can be evaluated by the $q$-Pfaff-Saalsch\"utz formula \eqref{qps}. We then get:
\begin{eqnarray*}
\beta_n(q^m,\alpha)&=&\frac{(\alpha)_{2n}(q^{m}/\alpha)_{n}}{(q^{1+m})_{2n}(\alpha q^{-m})_{n}}(\alpha q^{-m})^n\frac{(\alpha^2q^{-2m},q^{1-n})_{m+2n}}{(\alpha q^{1-n-m},\alpha q^{-m})_{m+2n}}
\end{eqnarray*}
which is the desired result.
\end{proof}
\begin{Remark}
When $\alpha\to0$ in Proposition~4.3, the WP-shifted Bailey pair becomes the shifted Bailey pair of Proposition~2.1.
\end{Remark}
As an application, we prove the following new transformation, which generalizes a result of Garrett, Ismail and Stanton \cite[(6.3)]{GIS}.
\begin{Theorem}
For all non negative integer $m$ and real parameters $\beta$, $\gamma$, $\rho$ such that $|q/\beta^2|<1$, we have: 
\begin{multline}\label{extension6.3}
\frac{(\beta)_m}{(q/\beta)_m}\sum_{n\in\mathbb{Z}}\frac{(1/\gamma,\rho,\gamma q^{1+m}/\beta\rho)_n}{(\gamma,q^{1+m}/\rho,\beta\rho/\gamma)_n}\frac{(\beta q^{m})_{2n}}{(q^{1+m}/\beta)_{2n}}(q/\beta)^{n}\\
=\frac{(q,q/\beta^2)_{\infty}}{(q/\beta,q/\beta)_{\infty}}\sum_{s=0}^{\left\lfloor m/2\right\rfloor}(\beta^3/q)^s\frac{(q/\gamma,\gamma q/\beta\rho,\rho q^{-m})_s}{(q,q/\rho,\beta\rho q^{-m}/\gamma)_s}\frac{1-\gamma q^{m-2s}}{1-\gamma}\\
\times\frac{(\beta,\gamma^2)_{m-2s}}{(q,\gamma q)_{m-2s}}\frac{(q)_{m-s}}{(\gamma q)_{m-s}}{}_4\phi_3\!\left[\begin{matrix}\beta/\gamma ,\beta q^{m-2s},\rho\beta q^{-s},\gamma q^{1+m-s}/\rho\\
\gamma q^{1+m-2s},\beta\rho q^{-s}/\gamma,q^{1+m-s}/\rho\end{matrix};
q,q/\beta^2\right].
\end{multline}
\end{Theorem}
\begin{proof}
Apply the first instance of Theorem~1.4 to Proposition~4.4 to get the WP-bilateral Bailey pair:
\begin{eqnarray*}
\alpha'_n(q^m,\alpha)&=&{(\rho_1,\rho_2)_n\over
(q^{1+m}/\rho_1,q^{1+m}/\rho_2)_n}(\alpha/c)^n\alpha_n(q^m,c)\\
&=&{(\rho_1,\rho_2)_n\over
(q^{1+m}/\rho_1,q^{1+m}/\rho_2)_n}(\alpha q^{-m})^n\frac{(q^m/c)_n}{(cq^{-m})_n}, 
\end{eqnarray*}
and
\begin{multline*}
\beta'_n(q^m,\alpha)=\frac{(\alpha\rho_1/q^m,\alpha\rho_2/q^m)_n}{(q^{1+m}/\rho_1,q^{1+m}/\rho_2)_n}\sum_{j\leq n}\frac{(\rho_1,\rho_2)_j}{(\alpha\rho_1/q^m,\alpha\rho_2/q^m)_j}\\
\times\frac{1-cq^{2j}}{1-c}
\frac{(\alpha/c)_{n-j}(\alpha)_{n+j}}{(q)_{n-j}(qc)_{n+j}}(\alpha/c)^j\beta_j(q^m,c),
\end{multline*}
with $c=\alpha\rho_1\rho_2/q^{1+m}$.\\
Now use the second instance of Theorem~1.4 to derive
$$
\widetilde{\alpha'}_n(q^m,\alpha)=\frac{(q^{1+2m}/\alpha)_{2n}}{(\alpha)_{2n}}\frac{(\rho_1,\rho_2,\alpha/\rho_1\rho_2)_n}{(q^{1+m}/\rho_1,q^{1+m}/\rho_2,\rho_1\rho_2/\alpha)_n}(\alpha q^{-m})^n
$$
and
$$
\widetilde{\beta'}_n(q^m,\alpha)=\sum_{j\leq n}\frac{(\alpha^2/q^{1+2m})_{n-j}}{(q)_{n-j}}(\alpha^2/q^{1+2m})^j\beta'_j(q^m,q^{1+2m}/\alpha).$$
Under the convergence condition $|\alpha^2/q^{1+2m}<1|$, let $n\to+\infty$ in the relation 
$$\widetilde{\beta'}_n(q^m,\alpha)=\sum_{r\leq n}\frac{(\alpha q^{-m})_{n-r}(\alpha)_{n+r}}{(q)_{n-r}(q^{1+m})_{n+r}}\,\widetilde{\alpha'}_r(q^m,\alpha).$$
This yields
\begin{eqnarray}
&&\hskip-3cm\sum_{n\in\mathbb{Z}}\frac{(q^{1+2m}/\alpha)_{2n}}{(\alpha)_{2n}}\frac{(\rho_1,\rho_2,\alpha/\rho_1\rho_2)_n}{(q^{1+m}/\rho_1,q^{1+m}/\rho_2,\rho_1\rho_2/\alpha)_n}(\alpha q^{-m})^n\nonumber\\
&=&\frac{(q^{1+m},\alpha^2/q^{1+2m})_{\infty}}{(\alpha q^{-m},\alpha)_{\infty}}\sum_{-\left\lfloor m/2\right\rfloor\leq s\leq j}f(s,j)\nonumber\\
&=&\frac{(q^{1+m},\alpha^2/q^{1+2m})_{\infty}}{(\alpha q^{-m},\alpha)_{\infty}}\sum_{s=0}^{\left\lfloor m/2\right\rfloor}\sum_{j\geq0}f(-s,j-s)\label{aide},
\end{eqnarray}
where 
\begin{multline*}
f(s,j):=\left(\frac{\alpha^2}{q^{1+2m}}\right)^j\left(\frac{\alpha q^{1+m}}{\rho_1^2\rho_2^2}\right)^s\times\frac{1-\rho_1\rho_2q^{m+2s}/\alpha}{1-\rho_1\rho_2q^{m}/\alpha}\left[{m+s\atop m+2s}\right]_q\\
\times\frac{(\rho_1q^{1+m}/\alpha,\rho_2q^{1+m}/\alpha)_j}{(q^{1+m}/\rho_1,q^{1+m}/\rho_2)_j}\times\frac{(\rho_1,\rho_2)_s}{(\rho_1q^{1+m}/\alpha,\rho_2q^{1+m}/\alpha)_s}\\
\times\frac{(q^{1+m}/\rho_1\rho_2)_{j-s}(q^{1+2m}/\alpha)_{j+s}}{(q)_{j-s}(\rho_1\rho_2q^{1+m}/\alpha)_{s+j}}\times\frac{(q)_m(\alpha q^{1-m}/\rho_1\rho_2)_{m-s}}{(\alpha q^{1-m}/\rho_1\rho_2,\rho_1\rho_2/\alpha)_{m}}\times\frac{(\rho_1^2\rho_2^2/\alpha^2)_{m+2s}}{(q\rho_1\rho_2/\alpha)_{m+s}}\cdot
\end{multline*}
Setting $\beta=q^{1+m}/\alpha$, $\gamma=\rho_1\rho_2/\alpha=\beta\rho_1\rho_2/q^{1+m}$, $\rho_1=\rho$, and finally rearranging the right-hand side of \eqref{aide}, we get \eqref{extension6.3}.
\end{proof}

\section{Concluding remarks}

We proved through extensions of the classical Bailey lemma many results from \cite{GIS}. It could be a challenging problem to find a proof  of the quintic formula \cite[Theorem~7.1]{GIS} through our approach.\\

We also want to point out that orthogonality relations and connection coefficient formulas for the $q$-Hermite and $q$-ultraspherical polynomials are at the heart of the proofs in \cite{GIS}. Recall that the $q$-ultraspherical polynomials have the explicit representation:
\begin{equation}\label{qultra}
C_n(\cos\theta;\beta|q)=\sum_{k=0}^n\frac{(\beta)_{k}(\beta)_{n-k}}{(q)_{k}(q)_{n-k}}\mbox{e}^{-i(n-2k)\theta}.
\end{equation}
One can see that \eqref{qultra} is equivalent to saying that $(\alpha_n(q^m,\beta q^m),\,\beta_n(q^m,\beta q^m))$ is a WP-shifted Bailey pair, where
$$\alpha_n(q^m,\beta q^m)=\mbox{e}^{2in\theta}\;\;\mbox{and}\;\;\beta_n(q^m,\beta q^m)=\mbox{e}^{-im\theta}\frac{(q)_m}{(\beta)_m}C_{2n+m}(\cos\theta;\beta|q).$$
Applying the first instance of Theorem~1.4 with $\rho_2\to+\infty$, $\rho_1\to0$ and $c=\beta \rho_1\rho_2/q<\infty$ yields the connection coefficient formula for $C_n$:
$$C_n(\cos\theta;c|q)=\sum_{k=0}^{\left\lfloor n/2\right\rfloor}\frac{(c/\beta)_k(c)_{n-k}}{(q)_k(q\beta)_{n-k}}\beta^k\frac{1-\beta q^{n-2k}}{1-\beta}C_{n-2k}(\cos\theta;\beta|q).$$
 A natural question is then to ask whether our method applied to the other Bailey  lemmas (classical, Well-Poised or elliptic from \cite{W3}) could prove or highlight properties for more general $q$-orthogonal polynomials.\\

 Besides, appart from the one mentioned in Section~3, there are many other changes of base in \cite{BIS}, which have a bilateral version. Moreover, there should be bilateral versions of the other WP Bailey lemmas proved by Warnaar \cite{W3} or Mc Laughlin and Zimmer \cite{MZ}. It could be interesting to derive applications of our method from all of these.\\

\noindent {\bf Aknowledgments} We would like to thank Mourad Ismail for interesting discussions during his visit in Lyon. We also thank very much Michael Schlosser and Ole Warnaar for very useful comments on an earlier version of this paper.


\end{document}